\documentclass[11pt]{amsart}
\usepackage{lineno}
\usepackage{amsmath}
\usepackage{amsfonts}
\usepackage{amssymb}
\usepackage{amsxtra}
\usepackage{amsthm}
\usepackage{amsmath,amssymb,amsfonts,amsthm,color,latexsym}
\usepackage{xcolor}
\usepackage{float}

\newcommand{\C}{\mathbb{C}}
\newcommand{\F}{\mathcal{F}}

\newcommand{\B}{\mathcal{B}}

\newcommand{\ord}{\text{ord}}

\newcommand{\sep}{\mathrm{Sep}}

\newtheorem{theorem}{Theorem}[section]
\newtheorem{example}[theorem]{Example}
\newtheorem{remark}[theorem]{Remark}

\newtheorem{proposition}[theorem]{Proposition}
\newtheorem{corollary}[theorem]{Corollary}

\newtheorem{maintheorem}{Theorem}

\usepackage[pagebackref]{hyperref}

\title[On some indices of foliations and applications]{On some indices of foliations and applications}

\date{\today}

\author[A. Fern\'andez-P\'erez]{Arturo Fern\'andez-P\'erez}
\address[Arturo Fern\'{a}ndez-P\'erez] {Department of Mathematics. Federal University of Minas Gerais. Av. Ant\^onio Carlos, 6627 
CEP 31270-901\\
Pampulha - Belo Horizonte - Brazil\\}
\email{fernandez@ufmg.br}

\author[E. R.  Garc\'{i}a Barroso]{Evelia R. Garc\'{i}a Barroso}
\address[Evelia R. Garc\'{i}a Barroso]{Departamento de Matem\'{a}ticas, Estad\'{\i}stica e Investigaci\'on Ope\-rativa.
Instituto Universitario de Matem\'aticas y Aplicaciones (IMAULL). Universidad de La Laguna. Apartado de Correos 456. 38200 La Laguna, Te\-ne\-rife, Spain.}
\email{ergarcia@ull.es}

\author[N. Saravia-Molina]{Nancy Saravia-Molina}
\address[Nancy Saravia-Molina]{Dpto. Ciencias - Secci\'{o}n Matem\'{a}ticas, Pontificia Universidad Cat\'{o}lica del Per\'{u}, Av. Universitaria 1801,
San Miguel, Lima 32, Peru}
\email{nsaraviam@pucp.edu.pe}



\begin{document}

\begin{abstract}
In this paper we establish a relationship between the Milnor number, the $\chi$-number, and the Tjurina number of a foliation with respect to an effective balanced divisor of separatrices. Moreover, using the G\'omez-Mont--Seade--Verjovsky index, we prove that the difference between the multiplicity and the Tjurina number of a foliation with respect to a reduced curve is independent of the foliation. We also derive a local formula for the Tjurina number of a foliation with respect to a reduced curve. From a global point of view, these results lead to the following consequences: we provide a new proof of a global result regarding the multiplicity of a foliation due to Cerveau-Lins Neto and a new proof of a Soares's inequality for the sum of the Milnor number of an invariant curve of a foliation. Additionally, we obtain bounds for the global Tjurina number of a foliation on the complex projective plane. Finally, we provide an answer to the conjecture posed by Alc\'antara and Mozo-Fern\'andez about foliations on the complex projective plane having a unique singularity.  
\end{abstract}
\maketitle

\section{Introduction}
The aim of this paper is to study indices of local and global foliations in the complex projective plane \(\mathbb{P}^2_{\mathbb{C}}\). We focus particularly on four indices of the foliations: the \textit{Milnor number}, the \textit{G\'omez-Mont--Seade--Verjovsky index} (abbreviated as the GSV-index), the \textit{multiplicity}, and the \textit{Tjurina number} with respect to an invariant curve. 
\par First, using known local formulas for the GSV-index and the Tjurina number of a singular curve, we present a local formula for the Tjurina number of a foliation (see Proposition \ref{prop_tju}). Next, we prove in Proposition \ref{prop_1} that the difference between the multiplicity and the Tjurina number of a foliation with respect to a reduced curve does not depend on the foliation. More specifically, let $\F$ be a germ of a singular holomorphic foliation at $(\C^2,p)$ and let $C$ be any $\F$-invariant reduced curve. Denoting by $\mu_p(\F,C)$ and $\tau_p(\F,C)$ the multiplicity and Tjurina number of $\F$ along $C$, respectively, we obtain
\[\mu_p(\F,C)-\tau_p(\F,C)=\mu_p(C)-\tau_p(C),\]
 \noindent where $\mu_p(C)$, $\tau_p(C)$ are the Milnor and Tjurina numbers of $C$. 
This result leads to several applications; for instance, we show that 
\[GSV_p(\F,C) < 4\tau_p(\F,C) - 3\mu_p(\F,C),\]
where $GSV_p(\F,C)$ denotes the $GSV$-index of $\F$ along $C$ at $p$.
Some applications also are presented in  Sections \ref{gsv_mul} and \ref{applications}, for example, we provide a positive answer to a question posed in \cite{FP-GB-SM2024}.
\par In Section \ref{multi}, following the terminology of balanced divisor of separatrices of a foliation introduced by Genzmer \cite{Genzmer}, we present Proposition \ref{iguales} that generalizes Corollary B of \cite{FP-GB-SM2021}. We recall that a foliation $\F$ is said to be a \textit{generalized curve} if there are no saddle-nodes in its reduction process of singularities. This concept was introduced by Camacho, Lins Neto and Sad in \cite{CLS} and the non-dicritical case, delimits a family of foliations whose topology is closely related to that of their separatrices (i.e. local invariant curves). Generalized curve foliations are part of the broader family of \textit{second type foliations}, introduced by Mattei and Salem in \cite{mattei}. Foliations in this family may admit saddle-nodes in the reduction process of singularities, provided that they are not \textit{tangent saddle-nodes} (see \cite[Definition 2.1]{FP-M}). In order to study second type foliations, the $\chi$-number of $\F$ was introduced in \cite[Section 3]{FP-GB-SM2021}. This number, denoted by $\chi_p(\F)$, can be interpreted as the difference $\mu_p(\F)-\mu_p(\F,\B)$, where $\mu_p(\F)$ is the Milnor number of $\F$ at $p$ and $\mu_p(\F,\B)$ is the multiplicity of $\F$ along a balanced divisor of separatrices $\B$. For more details on balanced divisor of separatrices of a foliation, we refer the reader to \cite{Genzmer} and \cite{Genzmer-Mol}.
\par Now, we can establish our main result as follows:  
\begin{maintheorem}\label{DG}
Let $\F$ be a germ of a singular holomorphic foliation  at $(\C^2,p)$, and $\B$ be an effective primitive balanced divisor of separatrices for $\F$ at $p$.  
  Then
    \[
    \frac{\mu_{p}(\mathcal{F})}{\tau_{p}(\mathcal{F},\mathcal{B})+\chi_{p}(\mathcal{F})}<\frac{4}{3}.
    \] 
Moreover, if  $\mathcal{F}$ is of second type then
$\frac{\mu_{p}(\mathcal{F})}{\tau_{p}(\mathcal{F},\mathcal{B})}<\frac{4}{3}.$ In particular,  if  $\mathcal{F}$ is a non-dicritical foliation of second type, and $C$ is the total union of separatrices, then $\frac{\mu_{p}(\mathcal{F})}{\tau_{p}(\mathcal{F},C)}<\frac{4}{3}.$
\end{maintheorem}
The proof of Theorem \ref{DG} will be given in Section \ref{principal}.
Finally, we apply our local results in three ways. First, we provide a new proof of a result by Cerveau-Lins Neto \cite[Proposition \ref{prop_cl}]{Cerveau-LN} and a new proof of a Soares's inequality \cite[Theorem \ref{Soares ine}]{Soares} for the sum of the Milnor number of an invariant curve of a foliation. In Theorem \ref{tjurina}, we establish bounds for the global Tjurina number of a foliation on the complex projective plane. Lastly, we answer a question posed by Alc\'antara and Mozo-Fern\'andez \cite[p. 16]{Alcantara} regarding the existence of non-dicritical logarithmic foliations with a unique singular point and an invariant algebraic curve passing through the singularity.

\section{Basic tools}
\par  To establish the terminology and notation, we recall  some basic concepts of the theory of holomorphic foliations. Unless stated otherwise, throughout this text, $\F$ denotes a germ of a singular holomorphic foliation at $(\C^2,p)$. In local coordinates $(x,y)$ centered at $p$, the foliation $\F$ is given by a holomorphic 1-form
\begin{equation*}
\omega=P(x,y)dx+Q(x,y)dy,
\end{equation*}
or by its dual vector field
\begin{equation*}
v = -Q(x,y)\frac{\partial}{\partial{x}} + P(x,y)\frac{\partial}{\partial{y}},
\end{equation*}
where  $P(x,y), Q(x,y)   \in \mathbb {C}\{x,y\}$ are relatively prime and ${\mathbb C}\{x,y\}$ is the ring of complex convergent power series in two variables. 
 The \textit{algebraic multiplicity} $\nu_p(\F)$ is the minimum of the orders $\nu_p(P)$, $\nu_p(Q)$ at $p$ of the coefficients of a local generator of $\F$. 
\par Let $f(x,y)\in  \mathbb{C}[[x,y]]$, where ${\mathbb C}[[x,y]]$ is the ring of complex formal power series in two variables.  
We say that $C: f(x,y)=0$  is {\em invariant} by $\F$ or $\F$-\emph{invariant} if $$\omega \wedge d f=(f.h) dx \wedge dy,$$ for some $h\in \mathbb{C}[[x,y]]$. If $C$ is an irreducible $\F$-invariant curve then we will say that $C$ is a {\em separatrix} of $\F$ at $p$. The separatrix $C$ is analytical if $f$ is convergent.  We denote by $\sep_p(\F)$ the set of all separatrices of $\F$ at $p$. When $\sep_p(\F)$ is a finite set we will say that the foliation $\F$ is {\em non-dicritical} and the union of all elements of $\sep_p(\F)$ is called {\em total union of separatrices} of $\F$ at $p$. Otherwise, we will say that $\F$ is a {\em dicritical} foliation.

\subsection{GSV-index}
The GSV-index was introduced by G\'omez-Mont-Seade-Verjovsky in \cite{GSV}. This index is a topological invariant for vector fields on singular complex varieties and is closely related to the Euler-Poincar\'e characteristic of the Milnor fiber (see, for instance, \cite[Theorem 3.1]{Callejas}). It also relates to the Schwartz index and local Euler obstruction through a proportionality theorem \cite[Theorem 3.1]{brasselet}. Below we will present Brunella’s definition of the GSV-index for one dimensional holomorphic foliations (see \cite{index}).
\par Let $\F: \omega=0$ be a singular foliation at $(\mathbb C^2,p)$. Let $C:f(x,y)=0$ be an $\F$-invariant curve, where $f(x,y)\in \C[[x,y]]$ is reduced. Then, as in the convergent case, there are $g,h\in \C[[x,y]]$ (depending on $f$ and $\omega$), with $f$ and $g$ and $f$ and $h$ relatively prime and a $1$-form $\eta$ (see \cite[Lemma 1.1 and its proof]{suwa1995}) such that 
\begin{equation}\label{eq:suwa}
g\omega=hdf+f\eta.
\end{equation}

Expression (\ref{eq:suwa}) was first proved by K. Saito \cite{Saito2}. Subsequently, A. Lins Neto \cite{neto} obtained the same result in the case where $f$ is irreducible, and T. Suwa \cite{suwa1995}, in the case where $f$ is reduced, showed that the functions $f,g,h$ can be chosen in such a way that each pair $(f,g)$ and $(f,h)$ is relatively prime.

The {\em GSV-index} of the foliation $\F$  at $(\C^2,p)$ with respect to an analytic $\F$-invariant  curve $C$ is
\begin{equation}
\label{eq:GSV}
GSV_p(\F,C)=\frac{1}{2\pi i}\int_{\partial C}\frac{g}{h}d\left(\frac{h}{g}\right),
\end{equation}
where $g,h\in \C\{x,y\}$ are from \eqref{eq:suwa}. 
\subsection{The multiplicity of a foliation along a separatrix.}
Let $\F$ be a germ of a singular foliation at $(\C^2,p)$ induced by the vector field $v$ and $B$ be a separatrix of $\F$ at $p$. Let $\gamma: (\C,0)\to(\C^2,p)$ be a primitive parametrization of $B$. Camacho-Lins Neto-Sad \cite[Section 4]{CLS} defined 
the \textit{multiplicity of $\F$ along $B$ at $p$} as
$\mu_p(\F,B):=\ord_t \theta(t)$,
where $\theta(t)$ is the unique vector field at $(\C,0)$ such that 
$\gamma_{*} \theta(t)=v\circ\gamma(t)$. 
If $\omega=P(x,y)dx+Q(x,y)dy$ is a 1-form inducing $\F$ and $\gamma(t)=(x(t),y(t))$, we have
\begin{equation*}
\theta(t)=
\begin{cases}
-\frac{Q(\gamma(t))}{x'(t)} & \text{if $x(t)\neq 0$}
\medskip \\
 \frac{P(\gamma(t))}{y'(t)} & \text{if $y(t)\neq 0$}.
\end{cases}
\end{equation*}
\subsection{The Tjurina number of a foliation along an invariant reduced curve}
Let $\F$ be the germ of holomorphic foliation defined by $\omega=P(x,y)dx+Q(x,y)dy$ at $(\C^2,p)$, and let $C: f(x,y)=0$ be an \textit{invariant reduced curve}. The \textit{Tjurina number} of $\F$ with respect to $C$ is by definition 
\[\tau_p(\F, C):=\dim_{\C}\C[[x,y]]/(P,Q,f).\]
This number appears for the first time in \cite{Gomez}, and the terminology of Tjurina number of $\F$ with respect to  $C$ was given in \cite[p. 159]{Cano}.

\section{Local formulas for the Tjurina number of a foliation}
Let $\F$ be a germ of a singular holomorphic foliation at $(\C^2,p)$, and let $C=\cup_{j=1}^{r}C_j$ be an $\F$-invariant reduced curve such that each $C_j$ is irreducible. Then, from \cite[Theorem 6.5.1]{Wall},  \cite[Proposition 3.2]{H-H}, \cite[Equation (2.2), p. 329]{K-S}, and \cite[p. 29]{Brunella-book}, we obtain the following formulas, respectively:
\begin{eqnarray}
\mu_p(C)&=&\sum_{j=1}^{r}\mu_p(C_j)+2\left(\sum_{1\leq i<j\leq r}I_p(C_i,C_j)\right)-r+1\label{eq_2},\\
\tau_p(C)&=&\sum_{j=1}^{r}\tau_p(C_j)+\left(\sum_{1\leq i<j\leq r}I_p(C_i,C_j)\right)+\Theta\label{eq_tjurina},\\
\mu_p(\F,C)&=&\sum_{j=1}^{r}\mu_p(\F,C_j)-r+1\label{eq_3},\\
GSV_p(\F,C)&=&\left(\sum_{j=1}^{r}GSV_p(\F,C_j)\right)-2\sum_{1\leq i<j\leq r}I_p(C_i,C_j),\label{eq_4}
\end{eqnarray}
where $I_p(C_i,C_j)$ denotes the intersection multiplicity of the branches $C_i$ and $C_j$ at $p$, and $\Theta$ is a number that depends on the module of K\"ahler differentials of $C$, whose explicit formula can be found in \cite[Theorem 5.1]{H-H}. Moreover, the \textit{Milnor number} $\mu_{p}(\F)$ of the foliation $\F$ at a point $p$, defined by the $1$-form $\omega=P(x,y)dx+Q(x,y)dy$, is given by 
$\mu_p(\F)=I_p(P,Q).$
Recall that we assume $P$ and $Q$ to be coprime, so $\mu_p(\F)$ is a non-negative integer. In \cite[Theorem A]{CLS}, it was proved that the Milnor number of a foliation is  a topological invariant. 

\par Analogous to the local formula for the Tjurina number of curves, we provide an local formula for the Tjurina number of a foliation $\F$ with respect to a reduced $\F$-invariant curve $C$, in terms of the Tjurina numbers of $\F$ with respect to the irreducible components of $C$.

\begin{proposition}\label{prop_tju}
Let $\F$ be a germ of a singular holomorphic foliation at $(\C^2,p)$ and let $C=\cup_{j=1}^{r}C_j$ be a $\F$-invariant reduced curve. Then
\begin{eqnarray*}
\tau_p(\mathcal{F},C)&=& \sum^{r}_{j=1}\tau_p(\mathcal{F},C_j)-\left(\sum_{1\leq i<j\leq r} I_p(C_i,C_j)\right)+\Theta.
\end{eqnarray*}
\end{proposition}
\begin{proof} After \cite[Proposition 6.2]{FP-GB-SM2021} we get $\tau_p(\mathcal{F},C)= GSV_p(\mathcal{F},C)+ \tau_p(C)$. Hence by equalities \eqref{eq_4} and \eqref{eq_tjurina}, we obtain
\begin{eqnarray*}
\tau_p(\mathcal{F},C)&=& GSV_p(\mathcal{F},C)+ \tau_p(C)\\
&=& \sum^{r}_{j=1}GSV_p(\mathcal{F},C_j)-2\sum_{i<j} I_p(C_i,C_j)+\sum^{r}_{j=1}\tau_p(C_j)+\left(\sum_{i<j} I_p(C_i,C_j)\right)+\Theta\\
&=& \sum^{r}_{j=1}\tau_p(\mathcal{F},C_j)-\left(\sum_{i<j} I_p(C_i,C_j)\right)+\Theta.
\end{eqnarray*}
\end{proof}

We illustrate Proposition \ref{prop_tju} with the following example.
\begin{example}\label{eje1}
We borrow from \cite[Example 6.5]{FP-GB-SM2021} the family of foliations $\F_k$, $k\geq 3$,  given by the $1$-form 
\[\omega_k=y\left(2x^{2k-2}+2(\lambda+1)x^2y^{k-2}-y^{k-1}\right)dx+x\left(y^{k-1}-(\lambda+1)x^2y^{k-2}-x^{2k-2}\right)dy,\]
which is a family of dicritical foliations which are not of second type. An effective  balanced divisor of separatrices of $\F_k$ is $\mathcal B=\mathcal B_0: xy=0$. Let $B_1:x=0$ and $B_2:y=0$. Hence $\tau_0(\F_k, \B)=3k-2$, $\tau_0(\F_k, B_1)=k$, and $\tau_0(\F_k,B_2)=2k-1$, since $\Theta=0$ by \eqref{eq_tjurina}, we get
\[\begin{array}{lll}
\tau_0(\mathcal{F}_k,\B) &=& \tau_0(\mathcal{F}_k,B_1)+ \tau_0(\mathcal{F}_K,B_2) - I_0(B_1,B_2)+\Theta\\
&=& k+2k-1-1+0=3k-2.
\end{array}
\]
\end{example}

\section{The GSV-index, the multiplicity and the Tjurina number of a foliation}\label{gsv_mul}
\par In this section, we present some local results involving  the GSV-index, the multiplici\-ty and the Tjurina number of a singular foliation $\F$ along an $\F$-invariant curve.
\begin{proposition}\label{prop_1}
Let $\F$ be a germ of a singular holomorphic foliation at $(\C^2,p)$ and let $C$ be any $\F$-invariant reduced curve. Then 
\begin{eqnarray}\label{eq_5}
GSV_p(\F,C)=\mu_p(\F,C)-\mu_p(C)
\end{eqnarray}
and 
\begin{eqnarray}\label{eq_impor}
\mu_p(\F,C)-\tau_p(\F,C)=\mu_p(C)-\tau_p(C).
\end{eqnarray}
\end{proposition}
\begin{proof}
Let $C=\cup_{j=1}^{r}C_j$. It follows from \cite[Eq. 5]{FP-GB-SM2024} that 
\begin{eqnarray}\label{eq_1}
GSV_p(\F,C_j)&=&\mu_p(\F,C_j)-\mu_p(C_j)\quad\text{ for all}\quad j=1,\ldots, r.
\end{eqnarray}
By combining \eqref{eq_1}, \eqref{eq_2}, \eqref{eq_3}, and \eqref{eq_4}, we obtain
$GSV_p(\F,C)=\mu_p(\F,C)-\mu_p(C)$. 
Finally, using \cite[Proposition 6.2]{FP-GB-SM2021}, we get
\[\mu_p(\F,C)-\tau_p(\F,C)=\mu_p(C)-\tau_p(C).\]
\end{proof}
After Proposition \ref{prop_1} and \cite[Theorem 3.2]{Almiron}, we prove Corollary \ref{coro-q} which provides a positive answer to Question 1 posed in \cite{FP-GB-SM2024}.
\begin{corollary}\label{coro-q}
Let $\F$ be a germ of a singular holomorphic foliation at $(\C^2,p)$ and let $C$ be any $\F$-invariant reduced curve. Then
\[GSV_p(\F,C)<4\tau_p(\F,C)-3\mu_p(\F,C).\]
\end{corollary} 
\begin{proof}
From Proposition \ref{prop_1}, equation \eqref{eq_impor}, and \cite[Theorem 3.2]{Almiron}, we obtain
\begin{equation*}\label{eq_r}
\mu_p(\F,C)-\tau_p(\F,C)=\mu_p(C)-\tau_p(C)<\mu_p(C)/4.
\end{equation*}
and proof follows from formula \eqref{eq_5}.
\end{proof}
\begin{corollary}
 Let $\F$ be a germ of a singular holomorphic foliation at $(\C^2,p)$ and let $C$ be a $\F$-invariant reduced curve. Then $\mu_p(\F,C)=\tau_p(\F,C)$   if and only if  after a suitable holomorphic
coordinate transformation $C$ is quasi-homogeneous.
\end{corollary}
\begin{proof} 
It follows from Proposition \ref{prop_1}, and \cite[Satz p.123, (a) and (d)]{Saito}.
\end{proof}
\begin{corollary}
Let $\F$ be a germ of a singular holomorphic foliation at $(\C^2,p)$ and let $C=\cup_{j=1}^{r}C_j$ be a $\F$-invariant reduced curve. Then 
\[\sum_{i=1}^{r}\left(\mu_p(\F,C_i)-\tau_p(\F,C_i)\right)\leq \mu_p(\F,C)-\tau_p(\F,C).\]
\end{corollary}
\begin{proof}
By \cite[Theorem 4.3]{H-H}, we have
\begin{equation}\label{ecc3}
\sum_{i=1}^{\ell}\left(\mu_p(C_i)-\tau_p(C_i)\right)\leq \mu_p(C)-\tau_p(C).
\end{equation}
Applying Proposition \ref{prop_1} to each $C_i$ and to $C$, and substituting into  \eqref{ecc3}, we obtain the desired inequality.
\end{proof}
\subsection{Values set of fractional ideals and the Tjurina number of a foliation}
In this subsection, we follow the notations from \cite{A-H}. Let $C:f(x,y)=0$ be a (reduced) germ of complex plane curve singularity at $p\in\C^2$ with equation $f=\Pi_{i=1}^{r}f_i=0$, where $f_i\in\C\{x,y\}$ is irreducible. Each $f_i$ defines a branch $C_i$ and its analytic type is characterized by the local ring $\mathcal{O}_i:=\C\{x,y\}/(f_i)$ up to $\C$-algebra isomorphism. The field of fractions $\mathcal{H}_i$ of $\mathcal{O}_i$ is isomorphic to $\C(t_i)$ and associated to it we have a canonical discrete valuation $\nu_i:\mathcal{H}_i\to\overline{\mathbb{Z}}:=\mathbb{Z}\cup\{\infty\}$. If $\overline{h}\in\mathcal{O}_i$ with $h\in\C\{x,y\}$, then we have $\nu_i(\overline{h})=I_p(f_i,h)$. The image of $\nu_i$ of $\mathcal{O}_i\setminus\{0\}$, that is
\begin{equation*}
S(C_i):=\{\nu_i(\overline{h})\in\mathbb{N}:\overline{h}\in\mathcal{O}_i,\,\,\,\overline{h}\neq 0\}
\end{equation*}
is the semigroup of values of $C_i$. Let  $\mathcal{O}:=\C\{x,y\}/(f)$. The semigroup of values  of $C$ is
\[
S(C):=\{(\nu_1(h),\ldots, \nu_r(h))\;:\; h\in \mathcal{O},\hbox{\rm $h$ is not a zero divisor}\}\subseteq \mathbb N^r.
\]
\par Let $\Omega^1=\C\{x,y\}dx+\C\{x,y\}dy$ be the $\C\{x,y\}$-module of holomorphic forms on $\C^2$ and consider the submodule $\F(f):=\C\{x,y\}df+f\Omega^1$. The module of K\"ahler differentials of $C$ is $\Omega_f:=\frac{\Omega^1}{\F(f)}$, that is the $\mathcal{O}$-module $\mathcal{O}dx+\mathcal{O}dy$ module the relation $df=0$. 
\par If $\varphi_i=(x_i,y_i)\in\C\{t_i\}\times\C\{t_i\}$ is a parametrization (non necessarily a Newton-Puiseux parametrization) of the branch $C_i$ and $h(x,y)\in\mathcal{O}$ then we denote $\varphi^{*}_i(h):=h(x_i,y_i)\in\C\{t_i\}$. In addition, given $\omega=A(x,y)dx+B(x,y) dy\in\Omega_f$, we set 
\[\varphi^{*}_i(\omega):=t_i\cdot (A(\varphi_i)\cdot x_i'+B(\varphi_i)\cdot y_i')\in\C\{t_i\},\]
where $x'_i,y_i'$ denote, respectively, the derivate of $x_i,y_i\in\C\{t_i\}$ with respect to $t_i$. Let $\mathcal{H}=\prod_{i=1}^r \mathcal{H}_i$ that is, the total ring of fractions of $\mathcal{O}$ and 
\[\varphi^{*}(\Omega_f)=\{(\varphi_1^{*}(\omega),\ldots,\varphi_r^{*}(\omega)):\omega\in\Omega_f\}\subset\mathcal{H}.\]
By \cite[Theorem 3.1]{Bayer}
\[\varphi^{*}(\Omega_f)\simeq\frac{\Omega_f}{Tor(\Omega_f)}\]
where $Tor(\Omega_f)$ is the torsion submodule of $\Omega_f$. In this way, $\varphi^{*}(\Omega_f)$ is a fractional ideal of $\mathcal{O}$ and, considering $\nu_i(\omega):=\nu_i(\varphi^{*}_i(\omega))$, its values set is given by 
\[\varLambda_f=\underline{\nu}(\varphi^{*}(\Omega_f))=\{\underline{\nu}(\omega):=(\nu_1(\omega),\ldots,\nu_r(\omega)):\omega\in\Omega_f\}.\]
We state the following theorem of \cite[Theorem 2.9]{A-H}.
\begin{theorem}
Let $C$ be a germ of reduced plane curve at $p\in\C^2$. With the previous notation, let $\overline{\varLambda}=\varLambda_f\cup \{0\}$. Then, we have
\[\mu_p(C)-\tau_p(C)=d(\overline{\varLambda}\setminus S(C)),\]
where $d$ denotes the distance function defined in \cite[Section 2.1]{A-H}.
\end{theorem}
In the context of foliations, Proposition \ref{prop_1} implies the following corollary.
\begin{corollary}
Let $\F$ be a germ of a singular holomorphic foliation at $(\C^2,p)$ and let $C:=f(x,y)=0$ be a $\F$-invariant reduced plane curve.  Then 
\[\mu_p(\F,C)-\tau_p(\F,C)=d(\overline{\varLambda}\setminus S(C)),\]
where $\overline{\varLambda}=\varLambda_f\cup \{0\}$.
\end{corollary}

\section{Multiplicity of a foliation along a divisor of separatrices}\label{multi} 
Following \cite[Section 2]{FP-GB-SM2024}, we define the multiplicity of  $\mathcal{F}$ along any non-empty divisor 
$\mathcal{B}=\sum_{B} a_B\cdot B$ of separatrices of $\mathcal F$ at $p$ as follows:
\begin{equation}\label{eq:gmul}
\mu_{p}(\mathcal{F},\mathcal{B})=\left(\displaystyle\sum_{B}a_{B}\cdot \mu_{p}(\mathcal{F},B)\right)-\deg(\mathcal{B})+1.
\end{equation}
The notion of balanced divisor of separatrices of a foliation was introduced by Genzmer in \cite{Genzmer}, see also \cite{Genzmer-Mol}.
By convention we set $\mu_{p}(\mathcal{F},\mathcal{B})=1$ for any empty divisor $\mathcal{B}$.
In particular, when $\mathcal{B}=B_1+\cdots+B_r$ is an \textit{effective primitive divisor of separatrices} of the singular foliation $\mathcal{F}$ at $p$, we recover \cite[Equation (2.2), p. 329]{K-S} (for the reduced plane curve $C:=\cup_{i=1}^rB_i$).
 Hence, if we write $\mathcal B=\mathcal B_0-\mathcal B_{\infty}$ where $\mathcal B_0$ and $\mathcal B_{\infty}$ are effective divisors we obtain
\begin{equation}\label{mudivisor}
\mu_{p}(\mathcal{F},\mathcal{B})=\mu_{p}(\mathcal{F},\mathcal{B}_0)-\mu_{p}(\mathcal{F},\mathcal{B}_{\infty})+1.
\end{equation}
\noindent Since $\mu(\mathcal F, \mathcal B_{\infty})\geq 1$, it follows that
$\mu_{p}(\mathcal{F},\mathcal{B}_0)\geq \mu_{p}(\mathcal{F},\mathcal{B}).$ 

\par In \cite[Section 3]{FP-GB-SM2021} the $\chi$-number of the foliation $\F$ at $p$  was introduced as $\chi_p (\mathcal{F})=\left(\sum_{q \in \mathcal I_p(\mathcal{F})}\nu_{q}(\mathcal{F})\xi_q(\F)\right)-\xi_p(\F)$, where $\mathcal I_p(\mathcal{F})$ denotes the set of infinitely near points of $\F$ at $p$ (see \cite[p. 7]{FP-GB-SM2021}), $\nu_{q}(\mathcal{F})$ is the algebraic multiplicity of the strict transform of $\mathcal{F}$ passing by $q$ and $\xi_p(\F)$ is the tangency excess of the foliation $\mathcal F$ (see \cite[Definition 2.3]{FP-GB-SM2021}).
By \cite[Proposition 4.7]{FP-GB-SM2021} we get 
    \begin{equation} \label{eq:IaMC}
    \mu_{p}(\mathcal{F},\mathcal B)=\mu_{p}(\mathcal F)-\chi_p(\mathcal{F}).
    \end{equation}
    Since  $\chi_p(\mathcal{F})$ is a non-negative number (see \cite[Proposition 3.1]{FP-GB-SM2021} ) we have
    \begin{equation}\label{eq:IaMineq}
\mu_{p}(\mathcal F)\geq \mu_{p}(\mathcal{F},\mathcal B).
    \end{equation}
Observe that, if $\mathcal B$ is  a primitive balanced divisor of separatrices for $\mathcal{F}$ at $p$ and $\mu_{p}(\mathcal{F},\mathcal B)=\mu_{p}(\mathcal F)$ then, by \cite[Proposition 3.1]{FP-GB-SM2021}),
$\mathcal F$ is of second type or has algebraic multiplicity equal to one. Moreover, if $\mathcal F$ is of second type and $\mathcal B$ is  a primitive balanced divisor of separatrices of $\mathcal{F}$ then $\mu_{p}(\mathcal{F},\mathcal B)=\mu_{p}(\mathcal F)$.
\par The following proposition generalizes \cite[Corollary B]{FP-GB-SM2021}.
\begin{proposition}\label{iguales}
Let $\F$ be a germ of a singular foliation  at $(\C^2,p)$. Let $\B=\B_0-\B_{\infty}$ be a primitive balanced divisor of separatrices for $\F$ at $p$.  Then,
 \begin{equation*}
 \mu_p(\F)-\tau_p(\F,\B_0)=\mu_p(\B_0)-\tau_p(\B_0)-\mu_p(\F,\B_{\infty})+\chi_p(\F)+1.
 \end{equation*}
 Moreover, $\mu_p(\F)=\tau_p(\F,\B_0)$ if and only if  \[\mu_p(\B_0)-\tau_p(\B_0)=\mu_p(\F,\B_{\infty})-\chi_p(\F)-1.\]
In particular, if $\B$ is an effective divisor then $\mu_p(\F)=\tau_p(\F,\B_0)$ if and only if  $\mu_p(\B_0)=\tau_p(\B_0)$ and $\chi_p(\F)=0$.
\end{proposition}
\begin{proof}
It follows from \eqref{eq:IaMC} and \eqref{mudivisor} that $\mu_{p}(\mathcal F)=\mu_{p}(\mathcal{F},\mathcal{B}_0)-\mu_p(\F,\B_{\infty})+\chi_p(\mathcal{F})+1$. The proof follows by applying Proposition \ref{prop_1} to $\B_0$, and substituting into the last equation.
Moreover, $\mu_p(\F)=\tau_p(\F,\B_0)$ if and only if \[\mu_p(\B_0)-\tau_p(\B_0)=\mu_p(\F,\B_{\infty})-\chi_p(\F)-1.\]
If $\B$ is effective, then the  last part of the proposition follows since $\mu_p(\B_0)-\tau_p(\B_0)$ and $\chi_p(\F)$ are non-negative integers.
\end{proof}
\par The following corollary generalizes \cite[Corollaries 3.4]{FS-GB-SM}.
\begin{corollary} \label{cor:iguales}
Let $\F$ be a germ of a singular holomorphic foliation  at $(\C^2,p)$. Let $\B$ be an effective primitive balanced divisor of separatrices. If $\mu_p(\F)=\tau_p(\F,\B)$ then, after a suitable holomorphic
coordinate transformation, $\B$ is quasi-homogeneous and either $\F$   has algebraic multiplicity $1$ at $p$ or $\F$ is of second type.
\end{corollary}
\begin{proof}
Since $\mu_p(\F)=\tau_p(\F,\B)$ and $\B$ is effective, by Proposition \ref{iguales} we obtain $\mu_p(\B)=\tau_p(\B)$. Hence, $\B$ belongs to its Jacobian ideal. The corollary follows from \cite[Proposition 3.1 (3)]{FP-GB-SM2021} and \cite[Satz p.123, (a) and (d)]{Saito}.
\end{proof}

\section{Proof of Theorem \ref{DG}}\label{principal}
In this section, we prove Theorem \ref{DG}. First, we present the following proposition.
\begin{proposition}\label{xi}
Let $\F$ be a germ of a singular foliation  at $(\C^2,p)$ and $\B=\B_0-\B_{\infty}$ be a balanced divisor of separatrices for $\F$ at $p$.  If $\mu_{p}(\mathcal{B}_{0}) \leq \mu_{p}(\mathcal{F})$,  then

\begin{equation}\label{ec1}
\frac{\mu_{p}(\mathcal{F})}{\tau_{p}(\mathcal{F},\mathcal{B}_{0})+\chi_{p}(\mathcal{F})-\mu_p(\F,\B_{\infty})+1}<\frac{4}{3}.
\end{equation}
\end{proposition}

\begin{proof}  
By \cite[Theorem 3.2]{Almiron}, we have
$
\mu_{p}(\mathcal{B}_{0})-\tau_{p}(\mathcal{B}_{0}) <  \frac{\mu_{p}(\mathcal{B}_{0})}{4}
$
so  after Proposition \ref{iguales} we get
\begin{eqnarray*}
\mu_{p}(\mathcal{F})-\tau_{p}(\mathcal{F},\mathcal{B}_{0})&<&\frac{\mu_{p}(\mathcal{B}_{0})}{4}-\mu_p(\F,\B_{\infty})+\chi_{p}(\mathcal{F})+1\\
&\leq&\frac{\mu_{p}(\mathcal{F})}{4}-\mu_p(\F,\B_{\infty})+\chi_{p}(\mathcal{F})+1,
\end{eqnarray*}
where we use the hypothesis for the last inequality. Hence
\[
3\mu_p(\F)-4\tau_p(\F,\B_0)<4\left(\chi_p(\F)-\mu_p(\F,\B_{\infty})\right)+4
\]
and the proposition follows.
\end{proof}
\begin{remark}
The hypothesis $\mu_{p}(\mathcal{B}_{0}) \leq \mu_{p}(\mathcal{F})$ of Proposition \ref{xi} is essential as the following example illustrates: we consider the Suzuki's foliation (see for instance \cite[p. 75]{Cerveau-Mattei}) given by \[\F:\omega=(2y^2+x^3)dx-2xy dy\] whose separatrices are $C:x=0$ and $C_c:y^2-cx^2-x^3=0$, where $c\in \C$. Consider the balanced divisor of separatrices $\B:=C+C_0+C_1-C_2$.
After some computations we have: $17=\mu_0(\B_0)\not\leq  \mu_0(\F)=5$, $\mu_0(\F,\B_{\infty})=3$, $\tau_0(\F,\B_0)=5$, and $\chi_{0}(\F)=0$ since $\F$ is a curve generalized foliation. So the left term of \eqref{ec1} is $\frac{5}{3}$ and Proposition \ref{xi} does not hold.
\end{remark}
\subsection{Proof of Theorem A} 
Since $\B$ is an effective, primitive, balanced divisor of separatrices for $\F$ at $p$, we have $\B_{\infty}=\emptyset$ and $GSV_{p}(\mathcal{F},\B)=\Delta_p(\mathcal{F},\B)$, where $\Delta_p(\mathcal{F},\B)$ denotes the polar excess number of $\mathcal F$ with respect to $\B$, and this number is non-negative because it is a sum of non-negative numbers (see \cite[Equality(3.12)]{FP-M}).  Hence, $GSV_p(\F,\B)=\mu_p(\F,\B)-\mu_p(\B)\geq 0$ by \eqref{eq_5}, which implies that $\mu_p(\F,\B)\geq \mu_p(\B)$.
From \eqref{eq:IaMineq}, we get $\mu_p(\F)\geq \mu_p(\F,\B)\geq \mu_p(\B)$. Thus, Theorem \ref{DG} is a consequence of Proposition \ref{xi}. $\square$

\begin{example}
Returning to  Example \ref{eje1}, the family $\F_k$ of dicritical foliations which are not of second type with $k\geq 3$. For the effective  balanced divisor of separatrices $\mathcal B=\mathcal B_0: xy=0$ of $\F_k$ we have $\tau_0(\F_k, \B_{0})=3k-2$, $\mu_0(\F_k)=k(2k-1)$ and $\chi_0(\F_k)=2(k-1)^{2}$. Theorem  \ref{DG} is verified.
\end{example}

\section{Applications to global foliations}\label{applications}
In this section, we work with global foliations on the complex projective plane $\mathbb{P}^2_\C$. Let $\F$ be a holomorphic foliation on $\mathbb{P}^2_\C$. The number of points of tangency counted with multiplicities, between $\F$ and a non-invariant $L\subset\mathbb{P}^2_\C$ is the \textit{degree} of the foliation, denoted by $d$. 
\par A holomorphic foliation $\F$ on $\mathbb{P}^2_\C$ of degree $d$, is a homogeneous $1-$ form $\omega=A(x,y,z)dx+B(x,y,z)dy+C(x,y,z)dz$ with $A, B$ and $C$ homogeneous polynomials of degree $d+1$ such that $xA+yB+zC=0$.

\subsection{New proof of a global result of Cerveau-Lins Neto}
In \cite[p. 885]{Cerveau-LN} Cerveau and Lins Neto established the following proposition. The main ingredients in their proof are the reduction of singularities of foliations and the Poincar\'e-Hopf formula. Here, we provide an alternative proof using our previous results and some known local formulas. 
\begin{proposition}(Cerveau-Lins Neto \cite{Cerveau-LN})\label{prop_cl}
Let $C$ be an irreducible curve on $\mathbb{P}^2_\C$ of degree $d_0$ and $\F$ be a holomorphic foliation of degree $d$ having $C$ as a separatrix. Then
\[2-2g(C)=\sum_p\sum_i\mu_p(\F,C_i)-d_0(d-1),\]
where $g(C)$ is the geometric genus of $C$ and the sum is taken over all local branches $C_i$ of $C$ passing through the singularities $p$ of $\F$ in $C$. 
\end{proposition}
\begin{proof}
For each singularity $p$ of $\F$ such that $p\in C$, we have by \eqref{eq_5},

\[\mu_p(\F,C)=GSV_p(\F,C)+\mu_p(C).\] 
Therefore, by \eqref{eq_3}
\begin{equation}\label{eq_sum}
\sum_p\sum_i\mu_p(\F,C_i)-\sum_p(r_p-1)=\sum_p GSV_p(\F,C)+\sum_p \mu_p(C),
\end{equation}
 where $r_p$ is the number of branches at $p$.
Since $C$ is irreducible, the geometric genus formula  gives
\begin{equation}\label{eq_g}
g(C)=(d_0-1)(d_0-2)/2-\sum_p \delta_p,
\end{equation}
where $\delta_p=\dim_\C\widetilde{\mathcal{O}_p}/\mathcal{O}_p$ is the co-length of the local ring $\mathcal{O}_p$ of $(C,p)$ in its normalization. It follows from the \textit{Milnor formula} \cite[Theorem 10.5]{milnor} that
\begin{equation}\label{eq_m}
\mu_p(C)=2\delta_p-r_p+1.
\end{equation}
 The proof follows by combining \eqref{eq_g}, \eqref{eq_m}, and the GSV-index formula \cite[Proposition 2.3]{Brunella-book}
\[\sum_p GSV_p(\F,C)=(d+2)d_0-d_0^2\]
and then substituting these expressions into \eqref{eq_sum}.
\end{proof}
\subsection{New proof of a global result of Soares}
In \cite{Soares}, Soares  studied the problem of bounding the degree of an algebraic curve $C$ invariant for a foliation $\F$ on $\mathbb{P}^2_\C$ in terms of the degree of $\F$, in order to give an answer to the
 \textit{Poincaré Problem}. For a recent account of this problem, we refer to \cite{Correa} and the references therein. For a local version of the Poincar\'e problem, see \cite{Fortuny}. 
\par Before stating the main result of this subsection, we need the following remark.

\begin{remark}\label{remar1}
    We note that $\tau_p(\F,C)\geq 1$ for all  singularity $p$ of $\F$. Indeed,  there would exist a point $q\in C$ which is a singularity of $\F$ with $\tau_q(\F,C)=0$, then,  we can take local coordinates centered at $q$ such that $q=(0,0)$, $\F$ is given by a 1-form $\omega_q=P(x,y)dx+Q(x,y)dy$ and $C=\{f(x,y)=0\}$. Since   $\tau_q(\F,C)=0$, it follows that $1\in(P,Q,f)$ so there exist $g,h,k\in\C\{x,y\}$ such that 
    \begin{equation}\label{eq:global}
    1=g\cdot P+h\cdot Q+k\cdot f.
    \end{equation}
    
    However, since $q\in C$ is a singularity of $\F$, we have $P(0,0)=Q(0,0)=f(0,0)=0$. Substituting these values into equation \eqref{eq:global} leads to a contradiction.
\end{remark}
\par M. Soares proved the following theorem using polar varieties and characteristic classes.
In \cite{machado3}, the author obtained a characterization of Soares’s formula (in higher dimensions), in terms of the Schwartz and GSV indices of the foliation. Also, in \cite{machado2}, the author obtained another generalization of Soares’s theorem for foliations on hypersurfaces with isolated singularities. In \cite{Machado1}, the author presents a more general version of Soares’s theorem in higher dimensions.
Here, we present a simpler proof based on the indices studied above. 
\begin{theorem}(Soares \cite[Theorem II]{Soares})\label{Soares ine}
Let $C$ be an irreducible curve of degree $d_0>1$ that is invariant by a foliation $\F$ on $\mathbb{P}^2_\C$ of degree $d\geq 2$ with isolated singularities. Then
\[d_0(d_0-1)-\sum_{p}(\mu_p(C)-1)\leq (d+1)d_0,\]
where the sum runs over all singularities $p$ of $C$.    
\end{theorem}
\begin{proof}
For each singularity $p$ of $\F$ such that $p\in C$, we have, by \eqref{eq_5},
\begin{equation}\label{eq_soa}
    \mu_p(C)=\mu_p(\F,C)-GSV_p(\F,C).
\end{equation} 
It follows from \eqref{eq_impor} that $\mu_p(\F,C)\geq\tau_p(\F,C)$, since $\mu_p(C)\geq\tau_p(C)$, and therefore $\mu_p(\F,C)\geq 1$ by Remark \ref{remar1}. In particular, from \eqref{eq_soa}, we obtain
\begin{eqnarray}\label{eq_soares}
    \mu_p(C)-1\geq -GSV_p(\F,C).
\end{eqnarray}
The proof follows by summing over all singularities of $\F$ in $C$ and using the GSV-index formula \cite[Proposition 2.3]{Brunella-book}
$\sum_p GSV_p(\F,C)=(d+2)d_0-d_0^2$,
and then substituting this equality into \eqref{eq_soares}.
\end{proof}
\subsection{Bounds of the global Tjurina number of a foliation}
Let $\F$ be a holomorphic foliation on $\mathbb{P}^2_\C$ of degree $d$, having an invariant reduced algebraic curve $C$ of degree $d_0$. We define the global Tjurina number of $\F$ with respect to $C$ as $\tau(\F,C):=\sum_p \tau_p(\F,C)$, which is the sum of the Tjurina number of $\F$ at all singularities $p$ of $\F$ on $C$. Similarly, we denote $\mu(C):=\sum_{p}\mu_p(C)$, the \textit{global Milnor number} of $C$, and $\tau(C):=\sum_p \tau_p(C)$, the \textit{global Tjurina number} of $C$.

\begin{remark}\label{cotamuuu}
Teissier's Proposition \cite[Chapter II, Proposition 1.2]{Teissier} implies that
\begin{equation}\label{ect}
    I_p(\mathcal{P}^{df},C)=\mu_p(C)+\nu_p(C)-1,
\end{equation}
 and by \cite[Theorem 4.18]{H} we get 
 \begin{equation}\label{eci}
  I_p(\mathcal{P}^{df},C) \geq  \nu_p(df)\cdot\nu_p(C)=(\nu_p(C)-1)\cdot\nu_p(C).   
 \end{equation} Using \eqref{ect} and \eqref{eci} we obtain $\mu_p(C)\geq (\nu_p(C)-1)^{2}.$     
\end{remark}

\par Now, we present bounds for the global Tjurina number of a foliation on $\mathbb{P}^2_\C$.
\begin{theorem}\label{tjurina}
Let $\F$ be a holomorphic foliation on $\mathbb{P}_{\C}^2$ of degree $d\geq 2$ having an invariant irreducible algebraic curve $C$ of degree $d_0\geq 2$. Then
\[\tau(\mathcal{F},C)=d_0(d-1)-2g(C)-\sum_p(r_p-1)-\mu(C)+\tau(C)+2,\]
where $g(C)$ is the geometric genus of $C$ and $r_p$ is the number of branches of $C$ at $p$.
Moreover, 
\[\tau(\mathcal{F},C) \leq d_0(d_0-3)+d(d+1)-2g(C)-\sum_p(r_p-1) - \sum_p (\nu_p(C)-1)^{2}+3,\] 
and if $C$ is not concurrent lines, then 
\[2-2g(C)+ \left[\frac{d_0}{2}\right]+d-d_0-\sum_p(r_p-1) \leq \tau(\mathcal{F},C).\]
\end{theorem}
\begin{proof}
For each singularity $p$ of $\F$ such that $p\in C$, we have 
\begin{equation*}
    \tau_p(\mathcal{F},C)=\mu_p(\mathcal{F},C)-\mu_p(C)+\tau_p(C)
\end{equation*}
by Proposition \ref{prop_1}. Then, using the definition of $\tau(\F,C)$, the equation \eqref{eq_3} and Proposition \ref{prop_cl}, we get
\begin{eqnarray}
   \tau(\mathcal{F},C)=\sum_p \tau_p(\mathcal{F},C)&=&\sum_p\mu_p(\mathcal{F},C)-\mu(C)+\tau(C)\nonumber\\
   &=& \sum_p\left(\sum_i\mu_p(\F,C_i)-r_p+1\right)-\mu(C)+\tau(C)\nonumber\\
   &=&  \sum_p\sum_i\mu_p(\F,C_i)-\sum_p(r_p-1)-\mu(C)+\tau(C)\nonumber\\
   &=&d_{0}(d-1)+2-2g(C)-\sum_p(r_p-1)-\mu(C)+\tau(C).\label{eq_tju}
\end{eqnarray}
 Moreover, by \cite[Corollary 3.4]{Shin}, if $C$ is not concurrent lines, we have
\begin{equation}\label{cotamu}
\mu(C)\leq (d_0-1)^{2}-\left[\frac{d_0}{2}\right].
\end{equation} 
Furthermore, it follows from \cite[Theorem 3.2]{plessis} that
\begin{equation}\label{cotatjurina}
(d_0-d-1)(d_0-1) \leq \tau(C) \leq (d_0-1)^{2} -d(d_0 -d-1).
\end{equation}
Finally, the theorem follows by substituting \eqref{cotamu}, \eqref{cotatjurina} and Remark \ref{cotamuuu} into \eqref{eq_tju}.
\end{proof}

\subsection{Alc\'antara--Mozo-Fern\'andez's conjecture}
Now, we address the following conjecture posed by Alc\'antara--Mozo-Fern\'andez \cite[p. 16]{Alcantara}:
{\it Does there exist a logarithmic, non-dicritical, and quasi-homogeneous foliation $\F$ on $\mathbb{P}^2_\C$, with a unique singular point $p$, having an invariant reduced algebraic curve $C$ passing through $p$? }

\par These authors indicate that they have some evidence suggesting that this type of foliation does not exist; however, they state that they do not yet have a solution to the problem. 
\par Recall that a foliation $\F$ on $\mathbb{P}^{2}_\C$ is \textit{logarithmic} if there exist homogeneous, irreducible polynomials $f_1,\ldots,f_m\in\C[x,y,z]$, of degrees $d_1,\ldots,d_m$, respectively, such that a  1-form describing the foliation is:
\[\Omega=f_1\cdot\ldots\cdot f_m\sum_{i=1}^{m}\lambda_i\frac{df_i}{f_i},\]
\noindent for some non-zero complex numbers $\lambda_1,\ldots,\lambda_m$ satisfying $\sum_{i=1}^{m}\lambda_i d_i=0$. Note that the foliation has degree $\deg(\F)=\sum_{i=1}^{m}d_i-2$. In particular, taking $C=\{f_1\cdots f_m=0\}$, we obtain 
\begin{equation}\label{eq_degree}
    \deg(C)=\deg(\F)+2.
\end{equation} 
A logarithmic foliation $\F$ on $\mathbb{P}^2_\C$ is always a generalized curve foliation, by \cite[Section 6]{FP-M}.
\par We also recall that a non-dicritical foliation $\F$ is \textit{quasi-homogeneous} at $p$, if the set of all separatrices of $\F$ passing by $p$ is a germ of curve given in some coordinates by a quasi-homogeneous polynomial, see \cite[p. 1657]{Genzmer2008}.
\par We propose the following answer in the case where $C$ is irreducible. Note that in this situation it is not necessarily assume that the foliation is quasi-homogeneous. 
\begin{theorem}
There does not exist a logarithmic, non-dicritical foliation $\F$ on $\mathbb{P}^2_\C$ of degree $d$, with a unique singular point $p$, having an invariant irreducible algebraic curve $C$ of degree $d_0\geq 2$ passing through $p$. 
\end{theorem}
\begin{proof}
Suppose, for contradiction, that there exists a logarithmic, non-dicritical foliation $\F$ on $\mathbb{P}^2_\C$ of degree $d$, with a unique singular point $p$, having an invariant irreducible algebraic curve $C$ of degree $d_0\geq 2$ passing through $p$.
Since $\F$ is logarithmic, it follows that $\F$ is a generalized curve foliation at $p$ and $d_0=d+2$ (see \cite[Section 6]{FP-M}). Moreover, since $\F$ is non-dicritical at $p$, we have
\begin{equation}\label{eq_c}
\mu_p(C)=\mu_p(\F)=d^2+d+1
\end{equation}
 by \cite[Theorem 4]{CLS} and \cite[p. 19]{Brunella-book}. Now, since $C$ is irreducible, we can apply  P\l oski's inequality \cite[Lemma 2.2]{Ploski} to $C$ at $p$, obtaining
 \begin{equation}\label{eq_d}
 \mu_p(C)\leq (d_0-1)(d_0-2).
 \end{equation}
 Combining \eqref{eq_c}, \eqref{eq_d}, and using $d=d_0-2$, we get
 \[d^2+d+1=d_0^2-3d_0+3\leq d_0^2-3d_0+2,\]
 implying $1\leq 0$, which is a contradiction. This completes the proof.
\end{proof}

Now, we present an answer to the conjecture posed by Alc\'antara and Mozo-Fern\'andez in the case where $C$ is reduced.
\begin{theorem}
There does not exist a logarithmic, non-dicritical, quasi-homogeneous foliation $\F$ on $\mathbb{P}^2_\C$ of degree $d$, with a unique singular point $p$, that admits an invariant reduced algebraic curve $C$ of degree $d_0\geq 2$, containing all separatrices passing through $p$. 
\end{theorem}
\begin{proof}
Suppose, by contradiction, that there exists a logarithmic, non-dicritical, quasi-homogeneous foliation $\F$ on $\mathbb{P}^2_\C$ of degree $d$, with a unique singular point $p$, and an invariant reduced algebraic curve $C$ of degree $d_0\geq 2$ which contains all the separatrices passing through $p$.
\par We fix the singular point at $p=[0:0:1]$, and consider local coordinates $(x,y)\in\C^2$ centered at $p$. Since $C$ contains all the separatrices passing through $p$, and $\F$ is non-dicritical at $p$, it follows that $C$ is the total union of the separatrices of $\F$ at $p$. Moreover, if the germ of $C$ at $p$ is given by $C_p=\{f(x,y)=0\}$, then, by hypothesis, $f$ is quasi-homogeneous.
Hence, after a suitable change of variables, we can write
\[f(x,y)=u\cdot x^{m}\cdot y^{n}\Pi_{j=1}^{s}(y^k-\zeta_{j} x^{q}),\]
where $u$ is a unit, $m,n\in\{0,1\}$, $\gcd(k,q)=1$, and $\zeta_j\in\C^{*}$ for all $j=1,\ldots,s$. 
Without loss of generality, we may assume that $k<q$. Now, we analyze the possible branches of $f$. 
First, note that $m=0$, because the intersection $\{x=0\}\cap \{ y^k z^{q-k}-\zeta_{j} x^{q}=0\}$ consists of two points, namely $[0:0:1]$ and $[0:1:0]$, which would yield two singularities of $\F$, contradicting our hypothesis. 
Additionally, $C$ cannot have two different branches of the form $C_j=y^k-\zeta_{j} x^{q}$; otherwise, taking two such branches, say $C_i=y^k-\zeta_{i} x^{q}$ and $C_j=y^k-\zeta_{j} x^{q}$, we obtain, by B\'ezout's theorem 
\[I_p(C_i,C_j)=kq=q^2,\] which implies that $k=q$, yielding a contradiction. Consequently, we conclude that $f(x,y)=uy(y^k-\zeta x^{q})$, where $\zeta\in\C^{*}$. In particular, $d_0=q+1$, and since $\F$ is logarithmic,
$d=d_0-2=q-1$ by (\ref{eq_degree}).  Moreover, since $\F$ is a non-dicritical generalized curve foliation, we obtain by \cite[p. 19]{Brunella-book} and \cite[Theorem 4]{CLS}
\[d^2+d+1=\mu_p(\F)=\mu_p(C)=\mu_p(y)+\mu_p(y^k-\zeta x^{q})+2q-1,\] which implies that $q=k+1$. Thus, we conclude that $f(x,y)=uy(y^{k}-\zeta x^{k+1})$. In particular, $f$ has two branches and contains all separatrices of $\F$ passing through $p$. Now, observing that in homogeneous coordinates $[x:y:z]$ of $\mathbb{P}^2_{\C}$, the curve $C=\{y(y^{k}z-\zeta x^{k+1})=0\}$ is an algebraic curve invariant by $\mathcal{F}$, and since $d_0=q+1=k+2$ and $d=q-1=k$, we conclude that $\mathcal{F}$ can be defined by a closed logarithmic 1-form with poles along $C$, by Cerveau's Theorem \cite[Theorem 1.4]{Cerveau}. That is, there exist $\lambda_1, \lambda_2 \in \C^{*}$ satisfying $\lambda_1 + (k+1)\lambda_2 = 0$, and
\[
\mathcal{F}:\quad\Omega = \lambda_1 \frac{dy}{y} + \lambda_2 \frac{d(y^{k}z - \zeta x^{k+1})}{y^{k}z - \zeta x^{k+1}}.
\]
However, since $\lambda_1 + (k+1)\lambda_2 = 0$, we can write $\Omega$ as:
\[
\Omega = \lambda_2 \frac{y^{k+1}}{(y^{k}z - \zeta x^{k+1})} d\left( \frac{y^{k}z - \zeta x^{k+1}}{y^{k+1}} \right),
\]
which implies that $\mathcal{F}$ admits a rational first integral, and consequently, has a unique dicritical singularity at $p = [0:0:1]$. This contradicts the assumption that $\F$ is a non-dicritical foliation at $p$.
\end{proof}

\noindent{\bf{Acknowledgments}}\\
The authors are grateful to Marcelo E. Hernandes and Hern\'an Neciosup Puican for their remarks.\\
They would also like to thank the anonymous referees whose comments on a previous version of this article helped to improve both the content and the presentation.\\

\noindent{\bf{Funding information}}\\
The authors gratefully acknowledge the support of Universidad de La Laguna (Tenerife, Spain), where part of this work was done: Spanish grants PID2019-105896GB-I00 funded by MCIN/AEI/10.13039/501100011033 and PID2023-149508NB-I00 funded by 
MICIU/AEI
/10.13039/501100011033 and by
FEDER, UE. This work was  also funded by the Direcci\'on de Fomento de la Investigaci\'on at the PUCP through grant DFI-2024-PI1100. 
The first author acknowledges support from CNPq Projeto Universal 408687/2023-1 "Geometria das Equa\c{c}\~oes Diferenciais Alg\'ebricas" and CNPq-Brazil PQ-306011/2023-9.\\

\noindent{\bf Data availability}\\
We do not analyze or generate any datasets, because our work proceeds within a theoretical approach.\\

\noindent{\bf Conflict of interest}\\
The authors declare that they have no conflict of interest.

\end{document}